\documentclass[a4paper,11pt]{amsart} 
\usepackage{amsmath,amsxtra,amssymb,latexsym, amscd,amsthm}
\usepackage[mathscr]{eucal}
\usepackage{mathrsfs}
\usepackage{bbm}
\usepackage{enumerate}
\usepackage{pict2e}
\usepackage{tikz}
\usepackage{graphicx}
\usepackage[a4paper]{geometry}
\usepackage{color}
\usepackage{hyperref}

\theoremstyle{plain}
\newtheorem{thm}{Theorem}[section]

\newtheorem{cor}[thm]{Corollary}
\newtheorem{lem}[thm]{Lemma}
\newtheorem*{euc*}{Euclidean division}
\newtheorem*{fek*}{Fekete's Lemma}
\newtheorem*{kin*}{Kingman's Subadditive Ergodic Theorem}
\newtheorem*{fur*}{Furstenberg-Kesten Theorem}
\theoremstyle{definition}

\newtheorem{rem}[thm]{Remark}

\newtheorem*{rem*}{Remark}

\renewcommand{\Re}{\operatorname{Re}}

\makeatletter
\newcommand*{\ov}[1]{%
  $\m@th\overline{\mbox{#1}}$%
}
\newcommand*{\ovA}[1]{%
  $\m@th\overline{\mbox{#1}\raisebox{3mm}{}}$%
}
\newcommand*{\ovB}[1]{%
  $\m@th\overline{\mbox{#1\rule{0pt}{3mm}}}$%
}
\newcommand*{\ovC}[1]{%
  $\m@th\overline{\mbox{#1\strut}}$%
}
\newcommand*{\ovD}[1]{%
  $\m@th\overline{\mbox{#1\vphantom{\"A}}}$%
}
\newcommand*{\ovE}[1]{%
  $\m@th\overline{\raisebox{0pt}[1.2\height]{#1}}$%
}
\newcommand*{\ovF}[1]{%
  $\m@th\overline{\raisebox{0pt}[\dimexpr\height+0.3mm\relax]{#1}}$%
}
\newcommand*{\ovG}[1]{%
  $\m@th\overline{\raisebox{0pt}[\dimexpr\height+1mm\relax]{#1\vphantom{A}}}$%
}
\makeatother

\newcommand{\R}{\mathbb{R}}
\newcommand{\C}{\mathbb{C}}

\DeclareSymbolFont{extraup}{U}{zavm}{m}{n}
\DeclareMathSymbol{\varheart}{\mathalpha}{extraup}{86}
\DeclareMathSymbol{\vardiamond}{\mathalpha}{extraup}{87}

\newcommand{\nwc}{\newcommand}
\nwc{\nwt}{\newtheorem}
\nwt{coro}{Corollary}
\nwt{ex}{Example}
\nwt{prop}{Proposition}
\nwt{defin}{Definition}

\nwc{\mf}{\mathbf} 
\nwc{\blds}{\boldsymbol} 
\nwc{\ml}{\mathcal} 

\nwc{\lam}{\lambda}
\nwc{\del}{\delta}
\nwc{\Del}{\Delta}
\nwc{\Lam}{\Lambda}
\nwc{\elll}{\ell}

\nwc{\IA}{\mathbb{A}} 
\nwc{\IB}{\mathbb{B}} 
\nwc{\IC}{\mathbb{C}} 
\nwc{\ID}{\mathbb{D}} 
\nwc{\IE}{\mathbb{E}} 
\nwc{\IF}{\mathbb{F}} 
\nwc{\IG}{\mathbb{G}} 
\nwc{\IH}{\mathbb{H}} 
\nwc{\IN}{\mathbb{N}} 
\nwc{\IP}{\mathbb{P}} 
\nwc{\IQ}{\mathbb{Q}} 
\nwc{\IR}{\mathbb{R}} 
\nwc{\IS}{\mathbb{S}} 
\nwc{\IT}{\mathbb{T}} 
\nwc{\IZ}{\mathbb{Z}} 
\def\bbbone{{\mathchoice {1\mskip-4mu {\rm{l}}} {1\mskip-4mu {\rm{l}}}
{ 1\mskip-4.5mu {\rm{l}}} { 1\mskip-5mu {\rm{l}}}}}
\def\bbleft{{\mathchoice {[\mskip-3mu {[}} {[\mskip-3mu {[}}{[\mskip-4mu {[}}{[\mskip-5mu {[}}}}
\def\bbright{{\mathchoice {]\mskip-3mu {]}} {]\mskip-3mu {]}}{]\mskip-4mu {]}}{]\mskip-5mu {]}}}}
\nwc{\setK}{\bbleft 1,K \bbright}
\nwc{\setN}{\bbleft 1,\cN \bbright}
 
\nwc{\va}{{\bf a}}
\nwc{\vb}{{\bf b}}
\nwc{\vc}{{\bf c}}
\nwc{\vd}{{\bf d}}
\nwc{\ve}{{\bf e}}
\nwc{\vf}{{\bf f}}
\nwc{\vg}{{\bf g}}
\nwc{\vh}{{\bf h}}
\nwc{\vi}{{\bf i}}
\nwc{\vI}{{\bf I}}
\nwc{\vj}{{\bf j}}
\nwc{\vk}{{\bf k}}
\nwc{\vl}{{\bf l}}
\nwc{\vm}{{\bf m}}
\nwc{\vM}{{\bf M}}
\nwc{\vn}{{\bf n}}
\nwc{\vo}{{\it o}}
\nwc{\vp}{{\bf p}}
\nwc{\vq}{{\bf q}}
\nwc{\vr}{{\bf r}}
\nwc{\vs}{{\bf s}}
\nwc{\vt}{{\bf t}}
\nwc{\vu}{{\bf u}}
\nwc{\vv}{{\bf v}}
\nwc{\vw}{{\bf w}}
\nwc{\vx}{{\bf x}}
\nwc{\vy}{{\bf y}}
\nwc{\vz}{{\bf z}}
\nwc{\bal}{\blds{\alpha}}
\nwc{\bep}{\blds{\epsilon}}
\nwc{\barbep}{\overline{\blds{\epsilon}}}
\nwc{\bnu}{\blds{\nu}}
\nwc{\bmu}{\blds{\mu}}
\nwc{\bet}{\blds{\eta}}

\nwc{\bk}{\blds{k}}
\nwc{\bm}{\blds{m}}
\nwc{\bM}{\blds{M}}
\nwc{\bp}{\blds{p}}
\nwc{\bq}{\blds{q}}
\nwc{\bn}{\blds{n}}
\nwc{\bv}{\blds{v}}
\nwc{\bw}{\blds{w}}
\nwc{\bx}{\blds{x}}
\nwc{\bxi}{\blds{\xi}}
\nwc{\by}{\blds{y}}
\nwc{\bz}{\blds{z}}

\nwc{\cA}{\ml{A}}
\nwc{\cB}{\ml{B}}
\nwc{\cC}{\ml{C}}
\nwc{\cD}{\ml{D}}
\nwc{\cE}{\ml{E}}
\nwc{\cF}{\ml{F}}
\nwc{\cG}{\ml{G}}
\nwc{\cH}{\ml{H}}
\nwc{\cI}{\ml{I}}
\nwc{\cJ}{\ml{J}}
\nwc{\cK}{\ml{K}}
\nwc{\cL}{\ml{L}}
\nwc{\cM}{\ml{M}}
\nwc{\cN}{\ml{N}}
\nwc{\cO}{\ml{O}}
\nwc{\cP}{\ml{P}}
\nwc{\cQ}{\ml{Q}}
\nwc{\cR}{\ml{R}}
\nwc{\cS}{\ml{S}}
\nwc{\cT}{\ml{T}}
\nwc{\cU}{\ml{U}}
\nwc{\cV}{\ml{V}}
\nwc{\cW}{\ml{W}}
\nwc{\cX}{\ml{X}}
\nwc{\cY}{\ml{Y}}
\nwc{\cZ}{\ml{Z}}

\nwc{\fA}{\mathfrak{a}}
\nwc{\fB}{\mathfrak{b}}
\nwc{\fC}{\mathfrak{c}}
\nwc{\fD}{\mathfrak{d}}
\nwc{\fE}{\mathfrak{e}}
\nwc{\fF}{\mathfrak{f}}
\nwc{\fG}{\mathfrak{g}}
\nwc{\fH}{\mathfrak{h}}
\nwc{\fI}{\mathfrak{i}}
\nwc{\fJ}{\mathfrak{j}}
\nwc{\fK}{\mathfrak{k}}
\nwc{\fL}{\mathfrak{l}}
\nwc{\fM}{\mathfrak{m}}
\nwc{\fN}{\mathfrak{n}}
\nwc{\fO}{\mathfrak{o}}
\nwc{\fP}{\mathfrak{p}}
\nwc{\fQ}{\mathfrak{q}}
\nwc{\fR}{\mathfrak{r}}
\nwc{\fS}{\mathfrak{s}}
\nwc{\fT}{\mathfrak{t}}
\nwc{\fU}{\mathfrak{u}}
\nwc{\fV}{\mathfrak{v}}
\nwc{\fW}{\mathfrak{w}}
\nwc{\fX}{\mathfrak{x}}
\nwc{\fY}{\mathfrak{y}}
\nwc{\fZ}{\mathfrak{z}}

\nwc{\tA}{\widetilde{A}}
\nwc{\tB}{\widetilde{B}}
\nwc{\tE}{E^{\vareps}}
\nwc{\tk}{\tilde k}
\nwc{\tN}{\tilde N}
\nwc{\tP}{\widetilde{P}}
\nwc{\tQ}{\widetilde{Q}}
\nwc{\tR}{\widetilde{R}}
\nwc{\tV}{\widetilde{V}}
\nwc{\tW}{\widetilde{W}}
\nwc{\ty}{\tilde y}
\nwc{\teta}{\tilde \eta}
\nwc{\tdelta}{\tilde \delta}
\nwc{\tlambda}{\tilde \lambda}
\nwc{\ttheta}{\tilde \theta}
\nwc{\tvartheta}{\tilde \vartheta}
\nwc{\tPhi}{\widetilde \Phi}
\nwc{\tpsi}{\tilde \psi}
\nwc{\tmu}{\tilde \mu}

\nwc{\To}{\longrightarrow} 

\nwc{\ad}{\rm ad}
\nwc{\eps}{\epsilon}
\nwc{\ep}{\epsilon}
\nwc{\vareps}{\varepsilon}

\def\ep{\epsilon}

\def\sq2{\sqrt{2}}

\def\t2{{\mathbb T}^2}
\def\s2{{\mathbb S}^2}

\def\R{\mathbb{R}}

\def\C{\mathbb{C}}

\nwc{\lap}{\bigtriangleup}
\nwc{\rest}{\restriction}
\nwc{\Diff}{\operatorname{Diff}}
\nwc{\diam}{\operatorname{diam}}
\nwc{\Res}{\operatorname{Res}}
\nwc{\Spec}{\operatorname{Spec}}
\nwc{\Vol}{\operatorname{Vol}}
\nwc{\Op}{\operatorname{Op}}
\nwc{\supp}{\operatorname{supp}}
\nwc{\Span}{\operatorname{span}}

\nwc{\dia}{\varepsilon}
\nwc{\cut}{f}
\nwc{\qm}{u_\hbar}

\def\hto0{\xrightarrow{\hbar\to 0}}

\def\rto0{\xrightarrow{r\to 0}}

\providecommand{\norm}[1]{\lVert#1\rVert}

\nwc{\la}{\langle}
\nwc{\ra}{\rangle}
\nwc{\lp}{\left(}
\nwc{\rp}{\right)}

\nwc{\bequ}{\begin{equation}}
\nwc{\be}{\begin{equation}}
\nwc{\ben}{\begin{equation*}}
\nwc{\bea}{\begin{eqnarray}}
\nwc{\bean}{\begin{eqnarray*}}
\nwc{\bit}{\begin{itemize}}
\nwc{\bver}{\begin{verbatim}}

\nwc{\eequ}{\end{equation}}
\nwc{\ee}{\end{equation}}
\nwc{\een}{\end{equation*}}
\nwc{\eea}{\end{eqnarray}}
\nwc{\eean}{\end{eqnarray*}}
\nwc{\eit}{\end{itemize}}
\nwc{\ever}{\end{verbatim}}

\title{Some relations between the spectra of simple and non-backtracking random walks}
\author{Nalini Anantharaman}
\usepackage{calc}
\usepackage{yhmath}
\usepackage{graphicx}

\makeatletter
\newlength{\temp@wc@width}
\newlength{\temp@wc@height}
\newcommand{\widecheck}[1]{%
  \setlength{\temp@wc@width}{\widthof{$#1$}}%
  \setlength{\temp@wc@height}{\heightof{$#1$}}%
  #1\hspace{-\temp@wc@width}%
  \raisebox{\temp@wc@height+2pt}[\heightof{$\widehat{#1}$}]%
     {\rotatebox[origin=c]{180}{\vbox to 0pt{\hbox{$\widehat{\hphantom{#1}}$}}}}%
}
\makeatother

\begin{document}


\begin{abstract}We establish some relations between the spectra of simple and non-backtracking random walks on non-regular graphs,
generalizing some well-known facts for regular graphs. Our two main results are 1) a quantitative relation between the mixing rates of the simple random walk and of the non-backtracking random walk 2) a variant of the ``Ihara determinant formula'' which expresses the characteristic polynomial of the adjacency matrix, or of the laplacian, as the determinant of a certain non-backtracking random walk with holomorphic weights.  \end{abstract}

\maketitle


\section{Results}
It has been noted by many authors that non-backtracking random walks on graphs looking locally like trees are simpler, from a combinatorial point of view, than usual random walks. This is for instance an ingredient of the proof of Alon's conjecture on random regular graphs by Friedman \cite{Fri08} or, more recently, by Bordenave \cite{Bor-new}. The study of the spectrum of non-backtracking random walks is also at the heart of the ``community detection'' problem and of the solution to the ``spectral redemption conjecture'' for various models of random graphs \cite{BLM}. Non-backtracking random walks are known to mix faster than the usual ones \cite{ABLS}. In \cite{PL}, the non-backtracking random walk is used to prove a cut-off phenomenon for the usual random walk on regular graphs. 
In geometric group theory, the ``cogrowth'' is directly related to the leading eigenvalue of non-backtracking random walks~: this is discussed in \cite{OW}, where a possibility to use this to extend the notion of co-growth to non-regular graphs is suggested.
The non-backtracking random walk may also be used as an analog of ``classical dynamics'' in the field of quantum chaos on discrete graphs~: see the papers by Smilansky \cite{Smi10, Smi07}, which were a source of inspiration for this work.
This note is a contribution to the study of the relation between the spectra of simple and non-backtracking random walks, for {\em{non-regular graphs}}. It originates in the work \cite{AS}, where we used non-backtracking random walks to study the quantum ergodicity problem for eigenfunctions of Schr\"odinger operators on non-regular expander graphs.

Let $G=(V, E)$ be a graph without multiple edges and self-loops. We assume that the degree $D(x)$ of a vertex $x$ is bounded above and below~: $2\leq D(x)\leq D$. In the results about spectral gaps, we actually have to assume that $D(x)\geq 3$. 
We are interested in relating the spectrum of various operators, such as the laplacian, the adjacency matrix, and certain weighted non-backtracking random walks. Such relations are well-known for {\em{regular graphs}} (i.e. those for which $D(x)$ is constant), and the goal of this note is to partially extend what is known to non-regular graphs. Our two main results are 1) a relation between the mixing rates of the simple random walk and of the non-backtracking random walk 2) a variant of the ``Ihara determinant formula'' \cite{Iha66-1, Iha66-2} which expresses the characteristic polynomial of the adjacency matrix, or of the laplacian, as the determinant of a certain non-backtracking transfer matrix with holomorphic coefficients.

\bigskip

We will write $y\sim x$ to mean that $y$ is a neighbour of $x$ in $G$.

The first operator we are interested in is the adjacency matrix $\cA$. It acts on $\IC^V$ by the formula $\cA f(x)=\sum_{y\sim x} f(y)$, and is self-adjoint on $\ell^2(V, u)$ if $V$ is endowed with the uniform measure $u$. 

We are also interested in the spectrum of the laplacian, defined by~:
\begin{eqnarray}\nonumber P: \IC^V&\To&  \IC^V\\
Pf(x)&=&\frac{1}{D(x)}\sum_{y\sim x} f(y).
\label{e:lapl}
\end{eqnarray}
If we endow the set of vertices $V$ with the measure $\pi(x)= D(x)$, then $P$ is self-adjoint on $\ell^2(V, \pi)$.
Note that this implies $\sum_x Pf(x)\pi(x)=\sum_x f(x)\pi(x).$ The space $\ell^2_o(V, \pi)$ of functions orthogonal to constants in $\ell^2(V, \pi)$ is preserved by $P$.

Let $B$ be the set of oriented edges, endowed with the uniform measure $U$ (each edge has weight $1$).
Denoting $Q(x)=D(x)-1$, the ``transfer operator'' is defined by
\begin{eqnarray}\nonumber \cS: \ell^2(B, U) &\To&  \ell^2(B, U)\\
\cS f(e)&=&\frac{1}{Q(o(e))}\sum_{e' \leadsto e} f(e')
\label{e:siso}
\end{eqnarray}
where $e' \leadsto e$ means that $t(e')=o(e)$ and $e$ is not the reverse of $e'$.
The operator $\cS$ is stochastic, it is the generator of the non-backtracking random walk. It is not self-adjoint, but we have
$\sum_e \cS f(e)=\sum_e f(e).$ This is equivalent to saying that $\cS^*:\ell^2(B, U) \To  \ell^2(B, U)$ is also stochastic. When $G$ is finite, this implies that $\cS$ preserves the space $\ell^2_o(B, U)$ of functions orthogonal to constants in $\ell^2(B, U)$.

We will finally use the non-stochastic operator
\begin{eqnarray}\nonumber \cB: \IC^B &\To&  \IC^B\\
\cB f(e)&=&\sum_{e' \leadsto e} f(e')
\label{e:B}
\end{eqnarray}

\bigskip

If the graph $G$ is finite and $(q+1)$-regular, then $\cA$ and $P$ (resp. $\cB$ and $\cS$) are the same operator up to a homothety, and there is an explicit relation between the characteristic polynomials of $\cA$ and $\cB$~:
\begin{equation}\label{e:Ihara}
\det(I^{|B|}-u\cB)=(1-u^{2})^{r-1}\det((1+u^2q)I^{|V|} -u\cA)
\end{equation}
where $r=|E|-|V|+1$ is the rank of the fundamental group.
This is the contents of the Ihara determinant formula \cite{Iha66-1, Iha66-2}, generalised in stages by Hashimoto, Bass and Kotani--Sunada \cite{Hash89, HashH89, Hash90-1, Hash92-1, Hash92-2, Bass, KS}. For finite non-regular graphs the relation reads
\begin{equation}\label{e:Ihara2}
\det(I^{|B|}-u\cB)=(1-u^{2})^{r-1}\det(I^{|V|} -u\cA +u^2Q)
\end{equation}
where $Q$ is the diagonal matrix with components $Q(x)=D(x)-1$. Note that the right-hand side in \eqref{e:Ihara2} is not directly related to the characteristic polynomial of $\cA$. The identity \eqref{e:Ihara2} relates eigenvalues (and eigenvectors) of $\cS$ to solutions of $(I^{|V|} -u\cA +u^2Q)v=0$, which are {\emph{not}} eigenvectors of $\cA$.  Our goal is twofold~:
\begin{itemize}
\item for finite graphs, compare the mixing rates of $\cS$ and $P$. For regular graphs, there is an exact relation between eigenvalues of $\cS$ and $P$, which implies that the spectral gap of $P$ on $\ell^2_o(V, \pi)$ is explicitly related to the spectral gap of $\cS$ on $\ell^2_o(B, U)$. However since $\cS$ is not self-adjoint, the knowledge of its spectral gap is not sufficient to determine its mixing rate, one needs to control the angles beween eigenvectors. This analysis, done in \cite{ALM} and with more details in \cite{PL}, uses the fact that the eigenvectors of $\cS$ are explicitly related to those of $P$. Such explicit relations are not available for non-regular graphs and we have to find a more general method;
\item extend formula \eqref{e:Ihara} to non-regular graphs in a way different of \eqref{e:Ihara2}, by finding an identity involving the characteristic polynomial of $\cA$ on the right-hand side.
\end{itemize}

The result about spectral gaps also holds for infinite graphs. We recover, in a less geometric but more quantitative way, the result of Ortner and Woess \cite{OW} saying that $P$ has a spectral gap on $\ell^2(V, \pi)$ if and only if the spectral radius of $\cS$ on $\ell^2(B, U)$ is strictly less than $1$.

We do not discuss the ``spectral gap'' of $\cA$ as this is not as properly defined as for $P$ (the top eigenvalue and eigenvector of $\cA$ are not explicit in general).

\subsection{Spectral gap and mixing rate for $P$ and $\cS$.\label{s:exp}}

Let us first assume that $G$ is finite, connected and non-bipartite. This is equivalent to assuming that $1$ is a simple eigenvalue of $P^2$. In other words, the spectrum of $P^2$ in $\ell^2_o(V, \pi)$ is contained in $ [0, 1-\beta]$, for some $\beta>0$ which measures the mixing rate of the simple RW on $G$.

Our first result is the following~:
\begin{thm} \label{t:mix}Assume that $G$ is finite and that $D(x)\geq 3$ for all $x$. Assume that the spectrum of $P^2$ on $\ell^2_o(V, \pi)$ is contained in $ [0, 1-\beta]$.
Then the spectrum of $\cS^{*2}\cS^2$ on $\ell^2_o(B, U)$ is contained in $[0, 1- c(D,\beta)]$, where $ c(D,\beta)$ depends only on $D$ and $\beta$, and is positive if $\beta$ is so. \label{p:sg1}
\end{thm}
Note that $\norm{\cS}_{\ell^2_o(B, U)\To \ell^2_o(B, U)}=1$, as $\norm{\cS f}=\norm{f}$ as soon as $f$ is a function on $B$ that is constant on edges having the same terminus. However, our theorem
says that $\norm{\cS^2}_{\ell^2_o(B, U)\To \ell^2_o(B, U)}\leq (1- c(D,\beta))^{1/2}$. The value of $ c(D,\beta)$ is given in \eqref{e:c}. 

\begin{cor} For all $n\geq 1$,
$$\norm{\cS^n}_{\ell^2_o(B, U)\To \ell^2_o(B, U)}\leq (1- c(D,\beta))^{\lfloor n/4\rfloor}.$$
\end{cor}
This gives the rate of mixing of the non-backtracking RW.

The converse is easier~: in the course of the proof, we will also see that
if the spectrum of $\cS^{*2}\cS^2$ on $\ell^2_o(B, U)$ is contained in $[0, 1- c]$, then the spectrum of $P^2$ on $\ell^2_o(V, \pi)$ is contained in $ [0, 1-D^{-2}c]$ (Remark \ref{r:impli}).

We were primarily interested in finite graphs in view of the application to quantum ergodicity \cite{AS}, but the result also holds for infinite graphs~:
\begin{thm} \label{t:cog}Assume that $G$ is infinite and that $D(x)\geq 3$ for all $x$.

(i) If the spectrum of $\cS^{*2}\cS^2$ on $\ell^2(B, U)$ is contained in $[0, 1- c]$, then the spectrum of $P^2$ on $\ell^2(V, \pi)$ is contained in $ [0, 1-D^{-2}c]$.

(ii) If the spectrum of $P^2$ on $\ell^2(V, \pi)$ is contained in $ [0, 1-\beta]$,
then the spectrum of $\cS^{*2}\cS^2$ on $\ell^2(B, U)$ is contained in $[0, 1- c(D,\beta)]$, where $ c(D,\beta)$ is given by \eqref{e:c}. 
\end{thm}

\begin{rem} It is well-known that $G$ is amenable iff the spectral radius of $P$ is $1$ (see \cite{Kes59-2, DodKen, Dod}). Thus Theorem \ref{t:cog} says that $G$ is amenable iff $\norm{\cS^2}_{\ell^2(B, U)\To \ell^2(B, U)}=1$.
It was proven before by Ortner and Woess that $G$ is amenable iff the spectral radius of $\cS$ is $1$ \cite{OW}. This can be recovered by our methods~: indeed, one direction results from our Theorem \ref{t:cog}; in the other direction, one can follow the same lines as in our Remark \ref{r:impli} to show that if $\norm{\cS^{*n}\cS^n}_{\ell^2(B, U)\To \ell^2(B, U)}<1$ then $\norm{P^{2(n-1)}}_{\ell^2(V, \pi)\To \ell^2(V, \pi)}<1$ (with an explicit bound).

Our method is very down-to-earth and gives, by basic manipulations, a quantitative relation between the spectral gap of $P$ and $\norm{\cS^2}_{\ell^2(B, U)\To \ell^2(B, U)}$. The method in \cite{OW} is less direct and more geometric~: it starts from the general fact that $G$ is amenable iff SOLG (the symmetrized oriented line graph) is amenable. And then it is shown that SOLG is amenable iff the spectral radius of $\cS$ is $1$.
 
\end{rem}

\subsection{Determinant relation\label{s:det}} We now assume that $G$ is finite.

Let $T=(V(T), E(T))$ be the universal cover of $G$~: $T$ is a tree, and there exists a subgroup $\Gamma$ of of automorphism group of $T$, acting without fixed points on $V(T)$, such that $G=\Gamma\backslash T$. Let $\tilde\cA$ be the adjacency matrix of $T$. The Green function on $T$ will be denoted by
$$ G(x, y;z)=\la \delta_x, (\tilde\cA-z)^{-1}\delta_y\ra_{\ell^2(V(T))}$$
for $z\in \IC\setminus \IR$.

 Given $v,w \in T$ with $v \sim w$, we denote by ${T}^{(v|w)}$ the tree obtained by removing from ${T}$ the branch emanating from $v$ that passes through $w$.
We define the restriction $H^{(v|w)}(x,y) = H(x,y)$ if $v,w \in {T}^{(v|w)}$ and zero otherwise. We then denote $G^{(v|w)}(\cdot, \cdot;z)$ the corresponding Green function.

Given $z \in \C \setminus \R$, $v\in V$, $w$ a neighbour of $v$ we denote
\[
G^z(v)=G(\tilde v,\tilde v;z)  \quad \text{and} \quad \zeta^{z}(w,v) = -G^{(\tilde v|\tilde w)}(\tilde v,\tilde v;z) \, .
\]
where $(\tilde v, \tilde w)$ is a lift of the edge $(v, w)$ in $T$. This definition does not depend on the choice of the lifts.
If $e=(w, v)\in B$, we also use the notation $G^z(e)=G(\tilde w,\tilde v;z) $, $\zeta^z(e)=\zeta^{z}(w,v)$. Note that $G^z(e)$ is invariant under edge-reversal, whereas $\zeta^z(e)$ is not. In the formula below, the function $\zeta^z$ on $B$ acts on $\IC^{B}$ as a multiplication operator.

\begin{thm}\label{t:det} For all $z\in \IC\setminus \IR$,
\begin{equation}\label{e:det} \prod_{e\in E}(- G^z(e))\cdot \det\left((\zeta^z)^{-1} I^{|B|}- \cB\right) =\det\left(z I^{|V|}-\cA\right) \cdot \prod_{x\in V}(-G^z(x))\end{equation}
\end{thm}
\begin{rem}In the case of a $(q+1)$-regular graph, $\zeta^z$ is a constant function, which solves the quadratic equation
\begin{equation}\label{e:zeta}
z =  q \zeta^{z} + \frac{1}{\zeta^{z}} \, 
\end{equation}
See Lemma \ref{lem:zetapot} below.
We also have $G^z(x)=\frac{ \zeta^{z}}{(\zeta^{z})^2-1}$ and  $G^z(e)=\frac{ (\zeta^{z})^2}{(\zeta^{z})^2-1}$ for all $x$ and all $e$. 
It can be checked that Theorem \ref{t:det} reduces to \eqref{e:Ihara} by setting $u=\zeta^z$. It is, however, different from \eqref{e:Ihara2} for non-regular graphs. Although extensions of the Ihara formula \eqref{e:Ihara} have been studied by many authors, the variant \eqref{e:det} seems to be new.

Note that \eqref{e:det} holds for any functions $\zeta^z$ that are solutions of the system of algebraic equations appearing in Lemma \ref{lem:zetapot}. These are by no means unique~: for instance, in the regular case, there are 2 solutions to equation \eqref{e:zeta}. It is nice, however, to know an explicit solution of this system that can be expressed in terms of Green functions.
\end{rem}
 
\begin{rem}
The theorem generalizes to the case where $\cA$ is replaced by a discrete ``Schr\"odinger operator'' of the form
$\cA_p + W : \IC^{V}\To \IC^{V}$, $(\cA_p + W) f(x)=\sum_{y\sim x} p(x, y)f(y) + W(x)f(x)$ where $W$ is a real-valued function on $V$, and $p$ is such that $p(x, y)=p(y, x)\in \IR$ and $p(x, y)\not=0$ iff $x\sim y$. The definitions of the Green functions $G^z$ and $\zeta^z$ should be modified (in the obvious manner) to incorporate the weights $p$ and the potential $W$. The definition of $\cB$ should be modified to
$$\cB_p f(e)= \sum_{e' \leadsto e} p(e')f(e')$$
and \eqref{e:det} becomes
\begin{equation*}\prod_{e\in E} \frac{(-G^z(e))}{p(e)}\cdot \det\left((\zeta^z)^{-1} I^{|B|}- \cB_p\right) =\det\left(z I^{|V|}-\cA_p-W\right) \cdot \prod_{x\in V}{(-G^z(x))}.\end{equation*}
 This remark, in particular, allows to cover the case of $\det(z I^{|V|}-P)$, noting that $P$ is conjugate to $\cA_p$ with $p(x, y)=(D(x)D(y))^{-1/2}$.

\end{rem}

 \bigskip

{\bf{Acknowledgements~:}} The author is supported by the Labex IRMIA and USIAS of Universit\'e de Strasbourg, and by Institut Universitaire de France.
This material is based upon work supported by the Agence Nationale de la Recherche under grant No.ANR-13-BS01-0007-01,
 .

I am very thankful to Mostafa Sabri for his careful reading and numerous useful comments on the manuscript.

 \section{Proof of Theorem \ref{t:mix}}

We start by noting that if $f\in \ell^2_o(V, \pi)$,
then
 \begin{equation} \label{l:sg}\frac12\sum_{x\in V}\frac{1}{D(x)}\sum_{y, y'\sim x}|f(y)-f(y')|^2\geq \beta \norm{f}^2_{\ell^2(V, \pi)}.\end{equation}

This just comes from the identity
\begin{eqnarray*}
\frac12\sum_{x\in V}\frac{1}{D(x)}\sum_{y, y'\sim x}|f(y)-f(y')|^2&=&\sum_x D(x) |f(x)|^2-\sum_{x}D(x)|Pf(x)|^2\\
&=&  \la f, (I-P^2) f\ra_{\ell^2(V, \pi)}
\end{eqnarray*}

To prove Proposition \ref{p:sg1}, we use the following decomposition of the space of functions on $B$~:
\begin{equation}\ell^2_o(B, U)= O(\ell^2_o(V, \pi))\oplus T(\ell^2_o(V, \pi)) \oplus (O(\ell^2_o(V, \pi))^\perp \cap T(\ell^2_o(V, \pi))^\perp).\label{e:decompo}\end{equation}
The space $O(\ell^2_o(V, \pi))$ is the image of $\ell^2_o(V, \pi)$ under the map
$Of(e)=f(o(e)).$ Thus $O(\ell^2_o(V, \pi))$ is the space of functions (orthogonal to constants) such that $f(e)$ depends only on the origin of $e$. Similarly, $T(\ell^2_o(V, \pi))$ is the space of functions such that $f(e)$ depends only on the terminus of $e$. It is the image of $\ell^2_o(V, \pi)$ under the map
$Tf(e)=f(t(e)).$ If $G$ is non-bipartite, we have $O(\ell^2_o(V, \pi))\cap T(\ell^2_o(V, \pi))=\{0\}$.
Note that each space $O(\ell^2_o(V, \pi)), T(\ell^2_o(V, \pi))$ is orthogonal to $\bbbone$ in $\ell^2(B, U)$, but that the two spaces are NOT orthogonal to each other. The space $O(\ell^2_o(V, \pi))^\perp \cap T(\ell^2_o(V, \pi))^\perp$ is of dimension $2|E| -2|V| +1=r-1$, where $r$ is the rank of the fundamental group of $G$.
By definition, it is the space of functions $f:B\To\IC$ such that, for all $x\in V$,
\begin{equation}\label{e:sums}\sum_{e, o(e)=x} f(e)=0\quad\mbox{ and } \sum_{e, t(e)=x} f(e)=0.\end{equation}

\begin{rem}In the infinite case, the proof of Theorem \ref{t:cog} will be similar, using the decomposition
\begin{equation*}\ell^2(B, U)= O(\ell^2(V, \pi))\oplus T(\ell^2(V, \pi)) \oplus (O(\ell^2(V, \pi))^\perp \cap T(\ell^2(V, \pi))^\perp).\end{equation*}
\end{rem}

\bigskip

To start the proof we use the Dirichlet identity for $f\in ~\ell^2(B, U)$:
$$\la f, (I-\cS^{*2}\cS^2)f\ra_{\ell^2(B, U)}=\frac12\sum_{e, e'}|f(e)-f(e')|^2 \cS^{*2}\cS^2(e, e').$$
Let us decompose $f$ according to \eqref{e:decompo}~: $f=F+G+H$ where $F\in O(\ell^2_o(V, \pi)), G\in T(\ell^2_o(V, \pi)), H\in (O(\ell^2_o(V, \pi))^\perp \cap T(\ell^2_o(V, \pi))^\perp)$.

We are first going to prove that
\begin{equation}\la f, (I-\cS^{*2}\cS^2)f\ra_{\ell^2(B, U)}\geq Q^{-4}\beta\norm{F+H}^2_{\ell^2(B, U)}\label{e:first}\end{equation}
where $Q=D-1$ (and, recall, $D$ is an upper bound on the degree).

In order to have $\cS^{*2}\cS^2(e, e')>0$, there must exist $e_1, e'_1, e_2\in B$ such that $e\leadsto e_1\leadsto e_2$ and $e'\leadsto e'_1\leadsto e_2$. Counting the number of possibilities, we see that $\cS^{*2}\cS^2(e, e')\geq (D(t(e))-2)Q^{-4}\geq Q^{-4}$ if $t(e)=t(e')$. Here we use the assumption that $D(t(e))\geq 3$. Thus,
\begin{multline*}\la f, (I-\cS^{*2}\cS^2)f\ra_{\ell^2(B, U)}\geq \frac{Q^{-4}}2\sum_{e, e' : t(e)=t(e')}|f(e)-f(e')|^2\\
\geq \frac{Q^{-4}}2\sum_{e, e': t(e)=t(e')}\frac{1}{D(t(e))}|f(e)-f(e')|^2\\
=\frac{Q^{-4}}2\sum_{e, e': t(e)=t(e')}\frac{1}{D(t(e))}|(F+H)(e)-(F+H)(e')|^2.\end{multline*}
Let us fix a vertex $x\in V$. Using the fact that $F$ depends only on the origin, and that $H$ satisfies \eqref{e:sums},
\begin{multline}\sum_{e, e',  t(e)=t(e')=x}|(F+H)(e)-(F+H)(e')|^2= \sum_{y, y'\sim x}|F(y)-F(y')|^2 \\+\sum_{e, e',  t(e)=t(e')=x}|H(e)-H(e')|^2
+4\Re\sum_{e, e',  t(e)=t(e')=x} \bar F(e)(H(e)-H(e'))\\
 = \sum_{y, y'\sim x}|F(y)-F(y')|^2 +\sum_{e, e', t(e)=t(e')=x}|H(e)|^2+|H(e')|^2
+4\Re\sum_{e, e',  t(e)=t(e')=x} \bar F(e)H(e)\label{e:GH}
\end{multline}
Summing now over $x$, and using the fact that $F$ and $H$ are orthogonal,
\begin{multline}\sum_{e, e',  t(e)=t(e')}\frac{1}{D(t(e))}|(F+H)(e)-(F+H)(e')|^2
=\sum_x D(x)^{-1}\sum_{y, y'\sim x}|F(y)-F(y')|^2\\ + 2\sum_{e} |H(e)|^2+ 4\Re\sum_{e} \bar F(e)H(e) \\
=\sum_x D(x)^{-1}\sum_{y, y'\sim x}|F(y)-F(y')|^2 + 2\sum_{e} |H(e)|^2\\
\geq \sum_x D(x)^{-1}\sum_{y, y'\sim x}|F(y)-F(y')|^2 +2 \norm{H}^2_{\ell^2(B, U)}\\
\geq 2\beta \norm{F}^2  +2 \norm{H}^2 \geq 2\beta\norm{F+H}^2.
\end{multline}
On the last line, we have used \eqref{l:sg}. This concludes the proof of \eqref{e:first}.

Now, let $G\in  T(\ell^2_o(V, \pi)$. Then again, by looking at what it means to have $\cS^{*2}\cS^2(e, e')>0$, we see that if $y, y'$ are two vertices such that $dist(y, y')=2$ (in other words, $y$ and $y'$ have a common neighbour $x$), then we can find edges $e, e'$ such that $t(e)=y, t(e')=y'$ and $\cS^{*2}\cS^2(e, e')\geq Q^{-4}$. Indeed, we may choose $e, e'$ such that $t(e)=y$ and $o(e)\not=x$, $t(e')=y'$ and $o(e')\not=x$, and $\cS^{*2}\cS^2(e, e')\geq (D(x)-2)Q^{-4}\geq Q^{-4}$.

Thus
\begin{multline*}\la G, (I-\cS^{*2}\cS^2)G\ra_{\ell^2(B, U)}= \frac12\sum_{e, e'}|G(e)-G(e')|^2 \cS^{*2}\cS^2(e, e')\\ \geq \frac{Q^{-4}}2\sum_{x}\sum_{y, y'\sim x}  |G(y)-G(y')|^2 \geq Q^{-4}\beta \norm{G}^2. 
\end{multline*}

\begin{rem}\label{r:impli} We can also write (using the fact that $\cS^{*2}\cS^2$ is stochastic)
\begin{multline*}\la G, (I-\cS^{*2}\cS^2)G\ra_{\ell^2(B, U)}=\frac12 \sum_{y, y' : d(y, y')=2}|G(y)-G(y')|^2\sum_{e, e': t(e)=y, t(e')=y'} \cS^{*2}\cS^2(e, e')
\\ \leq D^2 \frac12 \sum_{x\in V} \frac{1}{D(x)}\sum_{y, y' : y\sim x\sim y'}|G(y)-G(y')|^2 =  D^2 \la G, (I-P^2) G\ra_{\ell^2(V, \pi)}
\end{multline*}
and this proves part (i) of Theorem \ref{t:cog}.
\end{rem}

Let $A>1$ (to be chosen later, depending on $\beta$ and $D$). Let $f=F+G+H$ as before. Assume first that $A\norm{F+H}\geq \norm{G}$. Then by the triangular inequality $\norm{f}\leq (1+A)\norm{F+H}$.
In addition, as we have seen, 
\begin{equation}\la f, (I-\cS^{*2}\cS^2)f\ra_{\ell^2(B, U)}\geq Q^{-4}\beta\norm{F+H}^2_{\ell^2(B, U)}\geq  Q^{-4}\beta(1+A)^{-2}\norm{f}^2.\label{e:case1}\end{equation}
Otherwise, $A\norm{F+H}\leq \norm{G}$, and $\norm{f}\leq (1+A^{-1})\norm{G}$.
Noting that the operator norm of $I-\cS^{*2}\cS^2$ is less than $1$, we write
 for all $f=F+G+H$,
\begin{multline}\la f, (I-\cS^{*2}\cS^2)f\ra_{\ell^2(B, U)}\geq \la G, (I-\cS^{*2}\cS^2)G\ra_{\ell^2(B, U)}-2 A^{-1}\norm{G}^2 - A^{-2}\norm{G}^2\\
\geq (Q^{-4}\beta-3A^{-1})\norm{G}^2 \geq \frac{(Q^{-4}\beta-3A^{-1})}{(1+A^{-1})^2}\norm{f}^2.\label{e:case2}
\end{multline}
 
Choosing $A$ such that $A^{-1}=Q^{-4}\beta/6$, and gathering \eqref{e:case2} and \eqref{e:case1} we get the result with 
\begin{equation}\label{e:c}c(D, \beta)=\min \left( \frac{Q^{-4}\beta}{2(1+Q^{-4}\beta/6)^{2}},\;\frac{Q^{-4}\beta}{(1+6Q^4/\beta)^{2}} \right).\end{equation}


\section{Proof of the determinant relation}
The relations in the next lemma follow from the resolvent identity, and are proven (for instance) in \cite{AS}. For a vertex $v$ of $T$, $\mathcal{N}_v$ stands for the set of neighbouring vertices.

\begin{lem}                     \label{lem:zetapot}
For any $v \in V(T)$, $z = E+i\eta   \in \C^+  $, if we let $2m^{z}(v)=-\frac{1}{G(v, v;z)}$, we have
\[
z =  \sum_{u \sim v} \zeta^{z}(v, u)+2m^{z}(v) \quad \text{and} \quad z =  \sum_{u \in \mathcal{N}_v \setminus \{w\}} \zeta^{z}(v, u) + \frac{1}{\zeta^{z}(w, v)} \, .
\]
For any non-backtracking path $(v_0,\dots,v_k)$ in $T$,
\begin{equation}\label{e:GG}
G(v_0,v_k;z) = \frac{-\prod_{j=0}^{k-1} \zeta^{z}({v_{j+1}},v_j)}{2m^{z}(v_k)}=  \frac{-\prod_{j=0}^{k-1} \zeta^{z}({v_{j}},v_{j+1})}{2m^{z}(v_0)}.
\end{equation}
 Also, for any $w\sim v$, we have
\begin{equation}\label{e:reverse}
\zeta^{z}(w, v) = \frac{m(w)^{z}}{m(v)^{z}} \,\zeta^{z}(v, w) \, , \qquad 
\frac{1}{\zeta^{z}(w, v)} - \zeta^{z}(v, w) = 2m^{z}(v) \, ,
\end{equation}
 \end{lem}
 \begin{rem}We can note that \eqref{e:GG} may be written as
 $$ ((\zeta^z)^{-1} I^{|B|}- \cB)^{-1}(e, e')=\delta_{x=y}+\sum_{k=0}^{+\infty} \zeta^{z}(e') (\zeta^z \cB)^k(e, e')=   -2m^z(x)(\cA -z)^{-1}(x, y)$$
 for all $e, e'\in B$ and $x=o(e'), y=t(e)$.
 
 In the case of regular graphs, this is formula (2.4) in \cite{OW}, where it is attributed to Grigorchuk, with various proofs published by Woess, Szwarc \cite{W, Szw}, Northshield, Bartholdi \cite{North, Bart}.
 \end{rem}
\subsection{Operator relations}
In this section $z\in\IC^+$ is fixed, so we write $\zeta(x, y)$ instead of $\zeta^z(x, y)$, $m(x)$ instead of $m^z(x)$. If $e=(x, y)\in B$, we write $m_1(e)=m(x)$ and $m_2(e)=m(y)$. A function on $B$ defines a multiplication operator on $\IC^B$ (i.e. an operator which is diagonal in the canonical basis). We use the same notation for a function and the associated operator.

Let us introduce the notation
$$\cP f(x)=\frac1{D(x)}\sum_{y\sim x} f(x, y).$$
This is a projector on the space of functions depending only on the origin, which may be identified with $\ell^2(V, \pi)$, isometrically embedded into $\ell^2(B, U)$ by the map $\psi\mapsto O(\psi)$ defined in the previous section.

Let
$L=D (2m_1)^{-1}\cP$. Let
$$Hg(x)=\sum_{y, y\sim x} \frac{1}{2m(y)}\left(\zeta(y, x)g(y, x)-g(x, y)\right)$$
Theorem \ref{t:det} is based on the following exact relation~:
\begin{prop}
$$H \circ (\zeta^{-1} I- \cB)=(\cA-zI)\circ L$$
\end{prop}
\begin{proof}
Let $\phi=Lf$ and $g=-(\zeta^{-1} I- \cB) f$. The latter relation implies that for any $y\sim x$,
$$\phi(x)=(2m(x))^{-1}\left(f(x, y)+g(y, x)+\frac{f(y, x)}{\zeta(y, x)}\right).$$
We then calculate
$$\cA\phi(x)=\sum_{y, y\sim x}\phi(y)=\sum_{y, y\sim x}\frac{1}{2m(y)}\left(f(y, x)+g(x, y)+\frac{f(x, y)}{\zeta(x, y)}\right).
$$
We now use Lemma \ref{lem:zetapot} to write
\begin{multline*}\sum_{y, y\sim x}\frac{1}{2m(y)}f(y, x)\\
=\sum_{y, y\sim x}\frac{1}{2m(y)}\left( \zeta(y, x) 2m(x)\phi(x)-\zeta(y, x) f(x, y)-\zeta(y, x)g(y, x)\right)\\
=\sum_{y, y\sim x}  \zeta(x, y) \phi(x)+\sum_{y, y\sim x}\frac{1}{2m(y)}\left(  -\zeta(y, x) f(x, y)-\zeta(y, x)g(y, x)\right)\\
= (z-2m(x)) \phi(x)+\sum_{y, y\sim x}\frac{1}{2m(y)}\left(  -\zeta(y, x) f(x, y)-\zeta(y, x)g(y, x)\right).
\end{multline*}
Altogether,
\begin{multline*}\cA\phi(x)=(z-2m(x)) \phi(x)+\sum_{y, y\sim x}\frac{1}{2m(y)}\left( \left(\frac{1}{\zeta(x, y)} -\zeta(y, x)\right) f(x, y)-\zeta(y, x)g(y, x)+g(x, y)\right)
\\= (z -2m(x)) \phi(x)+\sum_{y, y\sim x}  f(x, y)+ \sum_{y, y\sim x}\frac{1}{2m(y)} \left(-\zeta(y, x)g(y, x)+g(x, y)\right)\\=(z-2m(x)) \phi(x)+2m(x)\phi(x)+ \sum_{y, y\sim x}\frac{1}{2m(y)} \left(-\zeta(y, x)g(y, x)+g(x, y)\right)\\
=z \phi(x)+\sum_{y, y\sim x} \frac{1}{2m(y)}\left(-\zeta(y, x)g(y, x)+g(x, y)\right)
\end{multline*}
which is the desired relation.
\end{proof}

We note that $Hg=-D\left(\cP((2m_2)^{-1}g)+\cP((2m_2)^{-1}\iota(\zeta g))\right)$ where $\iota f(x, y)=f(y, x)$
is the edge reversal involution. So $H$ itself is of the form 
$H= D\cP\circ K$ where $K=(2m_2)^{-1}(\iota\zeta-I)$.
We have proven that
$$D\cP\circ K\circ (\zeta^{-1} I- \cB)=(A-z)\circ (2m_1)^{-1} D\cP.$$
This is equivalent to the two relations 
$$\cP\circ K\circ (\zeta^{-1} I- \cB)\circ P= D^{-1}(A-z)\circ (2m_1)^{-1} D \cP$$
and 
$$\cP\circ K\circ (\zeta^{-1} I- \cB)\circ (I-\cP)=0.$$
The latter implies that $K\circ (\zeta^{-1} I- \cB)$ sends $\mathrm{Ker } \cP$ to itself.

We use the decomposition ${\IC}^B = \mathrm{Im } \cP \oplus \mathrm{Ker } \cP$. The two relations above tell us that
$$\det [ K\circ (\zeta^{-1} I- \cB)]=\det[ (A-z)\circ (2m_1)^{-1}] \times \det[ K\circ (\zeta^{-1} I- \cB)]_{\mathrm{Ker } \cP \To \mathrm{Ker } \cP}.$$
But $f\in \mathrm{Ker } \cP$ is equivalent to $\cB f=-\iota f$. Thus if $f\in \mathrm{Ker } \cP$,
$$K\circ (\zeta^{-1} I- \cB)f(x, y)=(2m(y))^{-1}f(x, y)\left(\zeta(y, x)-\frac{1}{\zeta(x, y)}\right)= -f(x, y).$$
Finally we obtain
\begin{multline*}\det K \det (\zeta^{-1} I- \cB)=\det (A-z)  \prod_{x\in V} (2m(x))^{-1} \,  (-1)^{\dim  \mathrm{Ker } \cP} = \det (A-z)  \prod_{x\in V} (-G(x))\,  (-1)^{\dim  \mathrm{Ker } \cP}
\\= \det (A-z)  \prod_{x\in V} (-G(x)) (-1)^{|B|-|V|}=\det (z-A)  \prod_{x\in V} (-G(x))
\end{multline*}
since $|B|=2|E|$ is even.

To prove the determinant relation, there remains to compute $\det K$. But $K$ is diagonal by blocks of size $2$ in the canonical basis of $\IC^{B}$. More precisely, denoting by $\delta_e$ the element of $\IC^{B}$ that takes the value $1$ on $e$ and $0$ elsewhere,
$$K\delta_e=-\frac{1}{2m(y)}\delta_e+\frac{\zeta(x, y)}{2m(x)}\delta_{\hat e}$$
if $e=(x, y)$. Thus
$$\det K=\prod_{e=\{x, y\}\in E}(2m(x))^{-1}(2m(y))^{-1}(1-\zeta(x, y)\zeta(y, x))=\prod_{e=\{x, y\}\in E} (-G(x, y)).$$
This yields the announced relation.

\bibliographystyle{plain}
\bibliography{biblio-lemasson} 

\end{document}